\documentclass[12pt,a4paper]{article}
\usepackage{amsfonts}
\usepackage{amsmath,latexsym,amssymb,amsfonts,amsbsy, amsthm}
\usepackage[usenames]{color}
\usepackage{xspace,colortbl,cite}
\allowdisplaybreaks
\newcommand{\beq}{\begin{equation}}
\newcommand{\eeq}{\end{equation}}
\newcommand{\ben}{\begin{eqnarray}}
\newcommand{\een}{\end{eqnarray}}
\newcommand{\beno}{\begin{eqnarray*}}
\newcommand{\eeno}{\end{eqnarray*}}
\newcommand{\dint }{\displaystyle\int }

\theoremstyle{plain}
\newtheorem{theorem}{Theorem}[section]

\newtheorem{proposition}[theorem]{Proposition}
\newtheorem{lemma}[theorem]{Lemma}
\newtheorem{remark}[theorem]{Remark}

\numberwithin{theorem}{section} \numberwithin{equation}{section}
\renewcommand{\theequation}{\thesection.\arabic{equation}}

\newcommand{\average}{{\mathchoice {\kern1ex\vcenter{\hrule height.4pt
width 6pt depth0pt} \kern-9.7pt} {\kern1ex\vcenter{\hrule height.4pt
width 4.3pt depth0pt} \kern-7pt} {} {} }}

\newtheorem{thm}{Theorem}[section]

\newtheorem{coro}[thm]{Corollary}

\def\R{\mathbb{R}}

\renewcommand{\phi}{\varphi}

\newcommand{\be}{\begin{equation}}
\newcommand{\ee}{\end{equation}}

\newcommand{\N}{\mathbb{N}}

\newcommand{\cH}{{\mathcal H}}

\newcommand{\dist}{{\rm dist}}
\newcommand{\supp}{{\rm supp}}

\newcommand{\weak}{\rightharpoonup}

\newcommand{\eps}{\varepsilon}

\renewcommand{\theequation}{\thesection.\arabic{equation}}

\renewcommand{\epsilon}{\varepsilon}

\begin{document}
\title{Symmetry breaking via Morse index for equations and systems of H{\'e}non-Schr\"odinger type}
\author{Zhenluo Lou$^a$, Tobias Weth$^b$\thanks{Corresponding author.},
Zhitao Zhang$^{a,}$\thanks{Supported by
National Natural Science Foundation of China 11325107, 11331010.}  \\
   {\small $^a$ Academy of Mathematics and Systems Science, the Chinese Academy of Sciences;}\\
   {\small Beijing 100190; and School of Mathematical Sciences,}\\
\small {University of Chinese Academy of Sciences, Beijing 100049, P.R. China}\\
      {\small $^b$ Institut f\"{u}r Mathematik, Goethe-Universit\"{a}t Frankfurt, Robert-Mayer-Str. 10,}\\
    {\small D-60629 Frankfurt a.M.,Gemany.}\\
     {\small E-mail: lzhenluoxin2008@163.com, weth@math.uni-fankfurt.de, zzt@math.ac.cn }}
\date{}
\maketitle
\begin{abstract}
We consider the Dirichlet problem for the Schr\"odinger-H\'{e}non system
$$
-\Delta u + \mu_1 u  = |x|^{\alpha}\partial_u F(u,v),\quad \qquad
 -\Delta v + \mu_2 v  = |x|^{\alpha}\partial_v F(u,v)
$$
in the unit ball $\Omega \subset \R^N, N\geq 2$, where $\alpha>-1$ is a parameter and $F: \R^2 \to \R$ is a $p$-homogeneous $C^2$-function for some $p>2$ with $F(u,v)>0$ for $(u,v) \not = (0,0)$. We show that, as $\alpha \to \infty$, the Morse index of nontrivial radial solutions of this problem  (positive or sign-changing) tends to infinity. This result is new even for the corresponding scalar H\'{e}non equation and extends a previous result by Moreira dos Santos and Pacella \cite{Moreira-dos-Santos-Pacella} for the case $N=2$. In particular, the result implies symmetry breaking for ground state solutions, but also for other solutions obtained by an $\alpha$-independent variational minimax principle.
\end{abstract}


\vskip 0.2cm


\renewcommand{\theequation}{\thesection.\arabic{equation}}
\setcounter{equation}{0}
\section{Introduction}
We consider the Dirichlet problem for the generalized H\'{e}non equation
\begin{equation}\label{1.4}
\left \{
 \begin{aligned}
-\Delta u+ \mu u &=|x|^\alpha |u|^{p-2}u&&\qquad \text{in $\Omega$,}\\
u&=0&&\qquad \text{on $\partial \Omega$,}
\end{aligned}
\right.
\end{equation}
and the corresponding problem for a Schr\"{o}dinger-H\'{e}non system
\begin{equation}\label{eq:henon-1-zh}
\left \{
 \begin{aligned}
 -\Delta u + \mu_1 u  &= |x|^{\alpha}\partial_u F(u,v)&&\qquad \text{in $\Omega$,}\\
 -\Delta v + \mu_2 v  &= |x|^{\alpha}\partial_v F(u,v)&&\qquad \text{in $\Omega$,}\\
u&=v=0&&\qquad \text{on $\partial \Omega$,}
    \end{aligned}
 \right.
\end{equation}
Here $\Omega \subset \R^N,N\geq 2$ is the unit ball, $\mu,\mu_1,\mu_2 \ge 0$, $p>2$, $\alpha>-1$ and $F: \R^2 \to \R$ satisfies the following assumption:
\begin{itemize}
\item[(F)] $F$ is of class $C^2$ on $\R^2$, homogeneous of degree $p>2$ and satisfies $F(u,v) >0$ for $(u,v) \in \R^2 \setminus \{0\}$.
\end{itemize}

We note that (\ref{1.4}) is merely called H\'enon equation in the case where $\mu=0$, and it has been introduced by H\'enon in \cite{Henon} in the context of astrophysics.
One of the first mathematical papers on this equation is due to Weiming Ni \cite{N}, who observed
that the presence of the weighted term leads to new critical exponents for
the non-existence of classical positive solutions. After Ni's work,
(\ref{1.4}) has been studied extensively in recent years. In \cite{C-P,S-W,S-S-W} the
authors study the existence of the ground state solutions of
\eqref{1.4} and their asymptotic behavior both for $\alpha>0$ fixed, $p \to 2^*$ and
$2 < p< 2^*$, $\alpha \to \infty$. Here $2^*$ is the critical Sobolev exponent given by
$2^*= \frac{2N}{N-2}$ for $N \ge 3$ and $2^* = \infty$ for $N=1,2$.
We also note that partial symmetry and symmetry-breaking results for ground state solutions of
the H\'{e}non equation were obtained in \cite{S-W}, while partial symmetry results for sign changing solutions were studied in \cite{Bartsch-Weth-Willem,van-Schaftingen-Willem}.

In the special case where $F(u,v) = \frac{a_1u^4+a_2v^4}{4}+ b \frac{u^2 v^2}{2}$ with constants $a_1,a_2 >0$, $b\ge 0$,
System (\ref{eq:henon-1-zh}) is a weighted version of the nonlinear Schr\"odinger system
\begin{equation}\label{1.2}
\left\{\begin{aligned} &-\Delta u+\mu_1 u=a_1u^3+b uv^2 &~~
\hbox{in}~ &~~\Omega,\\ &-\Delta v+\mu_2 v=a_2 v^3+b
u^2v  &~~\hbox{in}~ &~~ \Omega.\\
\end{aligned}\right.
\end{equation}
This system arises both in the context of nonlinear optics and of Bose-Einstein condensation and has been receiving extensive attention in recent years, see \cite{A-C,B-L-W-Z,D-W-Z-1,D-W-Z-2,D-W-Z-3,D-W-W,S-W,TV} and the
references therein. The majority of papers is concerned with $\Omega= \R^N$, but also the case of bounded domains $\Omega \subset \R^N$ has been studied together with Dirichlet boundary conditions. We remark in particular that
Sato-Wang \cite{S-W813} studied the limit system with
$b\rightarrow+\infty$, and they obtained the existence of multiple
solutions of the limit system.

Note that, if (\ref{1.2}) is considered with Dirichlet boundary conditions, then every solution $(u,v)$ with $u>0$ and $v>0$ in $\Omega$ is radial by Troy's symmetry result in \cite{T} based on the moving plane method. The same radiality result applies to the more general system (\ref{eq:henon-1-zh}) in the case where $\alpha=0$ and when the system  is cooperative, i.e., $\partial_{uv}F(u,v)>0$ for $u,v>0$. On the other hand, the moving plane method breaks down in the case $\alpha>0$ and symmetry breaking of ground state solutions is expected.

The notion of ground state solutions is defined in the case where $2<p< 2^*$. In this case, both problems (\ref{1.4}) and (\ref{eq:henon-1-zh}) have a variational structure with respect to the Sobolev space $\cH:=H^1_0(\Omega)$, as solutions are critical points of the corresponding functionals
$$
I_h: \cH \to \R, \qquad I_h(u)= \frac{1}{2}\int_\Omega (|\nabla u|^2 + \mu u^2)\,dx - \frac{1}{p}\int_\Omega |x|^\alpha |u|^p\,dx
$$
and $I_{hs}: \cH\times \cH \to \R$ given by
$$
 I_{hs}(u,v)=\frac{1}{2}\dint_\Omega(|\nabla u|^2+\mu_1 u^2+|\nabla v|^2+\mu_2 v^2)dx-\frac{1}{2}\dint_\Omega |x|^\alpha F(u,v)dx,
$$
The corresponding {\em Nehari manifolds} are then given by
$$
\mathcal{N}_h:=\{u \in \cH \setminus \{0\} \::\: I_h'(u)u = 0\}
$$
and
$$
\mathcal{N}_{hs}:=\{(u,v) \in \cH \times \cH \setminus \{(0,0)\} \::\: I_{hs}'(u,v)(u,v) = 0\},
$$
and they form natural constraints in the sense that solutions of (\ref{1.4}) resp. (\ref{eq:henon-1-zh}) are automatically contained in
$\mathcal{N}_h$, $\mathcal{N}_{hs}$, respectively.

As remarked above, it is expected that, for $\alpha>0$ large, ground state solutions of (\ref{1.4}) resp. (\ref{eq:henon-1-zh}) are not radially symmetric. For the case of (\ref{1.4}) with $\mu=0$, this has already been proved in \cite{S-S-W}. There are basically two approaches to prove symmetry breaking, i.e., the non-radiality of ground state solutions of (\ref{1.4}) and (\ref{eq:henon-1-zh}) for $\alpha$ large. The first approach is based on direct energy comparison between radial and nonradial functions in the Nehari manifolds $\mathcal{N}_{h}$ and
$\mathcal{N}_{hs}$. The second approach is to use the Morse index, which is equal to one for every minimizer of $I_h$ on $\mathcal{N}_{h}$ and every minimizer of $I_{hs}$ on $\mathcal{N}_{hs}$.
This approach is in fact much more general since the Morse index of classical solutions of (\ref{1.4}) and (\ref{eq:henon-1-zh}) can be defined for arbitrary $p>2$. Moreover, Morse index estimates are available not only for ground state solutions  but also for critical points associated with more general minimax principles.

To define the Morse index, we note that, for a solution $u$ of (\ref{1.4}), the linearized operator at $u$ is given by
$$
L_{u}^\alpha \phi:= -\Delta \phi + \mu \phi - (p-1) |x|^\alpha |u|^{p-2},
$$
Here and in the following, when we refer to a solution of (\ref{1.4}) or of (\ref{eq:henon-1-zh}), we always mean a classical solution in $C^2(\overline \Omega)$. Then the operator $L_{u}^\alpha$ is self-adjoint in $L^2(\Omega)$ with domain
$H^2(\Omega) \cap H^1_0(\Omega)$ and form domain $H^1_0(\Omega)$, and the Morse index $\mu(u)$ of $u$ is defined as the number of negative
eigenvalues of $L^\alpha_u$ counted with multiplicity. Similarly, for a (non-singular) solution $(u,v)$ of (\ref{eq:henon-1-zh}), the Morse index $\mu(u,v)$ is defined as the number of negative eigenvalues of the linearized operator $L_{u,v}^\alpha$ given by
\begin{equation}
  \label{eq:linearized-operator}
L_{u,v}^\alpha  {\phi_1 \choose \phi_2 } := -\Delta  { \phi_1 \choose \phi_2 } +
{\mu_1 \phi_1 \choose \mu_2 \phi_2 } -  |x|^{\alpha} D^2 F(u,v) {\phi_1 \choose \phi_2 }
\end{equation}
with $D^2 F(u,v) =  {\partial_{uu}F(u,v)\;\;
\partial_{uv} F(u,v)   \choose \partial_{uv} F(u,v)\;\; \partial_{vv} F(u,v)}$. We note that $L_{u,v}^\alpha$ is self-adjoint in $L^2(\Omega,\R^2)$ with domain $H^2(\Omega,\R^2) \cap H^1_0(\Omega,\R^2)$ and form domain $H^1_0(\Omega,\R^2)$

The main result of the present paper is the following:

\begin{thm}\label{th1.4}
Let $p>2$.
\begin{enumerate}
\item[i)] We have $\mu_H(\alpha) \to \infty$ as $\alpha \to \infty$, where
$$
\mu_{h}(\alpha) := \inf \{\mu(u)\::\: \text{$u$ is a nontrivial radial solution of (\ref{1.4})}\}.
$$
\item[ii)] Suppose that (F) is satisfied, and let
$$
\mu_{hs}(\alpha) := \inf \{\mu(u,v)\::\: \text{$(u,v)$ is a nontrivial radial solution of (\ref{eq:henon-1-zh})}\}
$$
for $\alpha >0$. Then $\mu_{hs}(\alpha) \to \infty$ as $\alpha \to \infty$.
\end{enumerate}
\end{thm}

We remark that assertion (i) is in fact a consequence of assertion (ii). Indeed, if $p>2$ and $u$ is a solution of (\ref{1.4}), then $(u,0)$ is a solution of (\ref{eq:henon-1-zh}) with $\mu_1= \mu$, $\mu_2=0$ and the nonlinearity $F(u,v)= \frac{|u|^{p}+|v|^p}{p}$ which satisfies assumption $(F)$. Moreover, if $u$ has Morse index $\mu(u)= k$, then $(u,0)$ has  Morse index $\mu(u,0)=k$, since the linearized operator $L_{u,0}^\alpha$ coincides with ${L_u^\alpha \choose -\Delta }$ on $H^2(\Omega,\R^2) \cap H^1_0(\Omega,\R^2)$ and the second component is a positive semidefinite operator.\\

Theorem~\ref{th1.4} is new already in the special case of the H\'{e}non equation
(\ref{1.4}) with $\mu=0$. In this case, it extends and complements recent interesting results of Moreira dos Santos and Pacella in  \cite{Moreira-dos-Santos-Pacella}, who obtained explicit lower bounds on the Morse index in the planar case $N=2$. More precisely, they consider the equation $-\Delta u = |x|^\alpha f(u)$ in a ball or an annulus in $\R^2$ together with Dirichlet boundary conditions, and they prove that, for
{\em even $\alpha>0$}, radial sign changing solutions have a Morse index greater than or equal to $\alpha+3$, see \cite[Theorem 1.4]{Moreira-dos-Santos-Pacella}. Moreover, if $f$ is superlinear, the lower bound improves to $\alpha + n(u) + 3$, where $n(u)$ denotes the number of nodal domains of $u$. These results are obtained by means of special transformations and a study of the corresponding
non-weighted reduced problem. For this approach, the assumptions $\alpha$ even, $N=2$ are key requirements, and it also does not seem to extend to systems of type (\ref{eq:henon-1-zh}).

Theorem~\ref{th1.4} immediately implies the following result on ground state solutions, as the Morse index equals one for all ground state solutions of (\ref{1.4}) and (\ref{eq:henon-1-zh}).

\begin{coro}\label{thm 3.1}
Let $2<p<2^*$. Then there exists a number $\bar{\alpha}>0$ such that for $\alpha \ge \bar{\alpha}$
\begin{enumerate}
\item[i)] every
ground state solution of (\ref{1.4}) is not radially symmetric;
\item[ii)] every ground state solution of (\ref{eq:henon-1-zh}) is not radially symmetric.
\end{enumerate}
\end{coro}

As noted already, Corollary~\ref{thm 3.1}i) is due to \cite{S-S-W} in the case $\mu=0$. Moreover, it has been proved by Smets and Willem in \cite{S-W} that ground state solutions of (\ref{1.4}) are foliated Schwarz symmetric for every
$\alpha \ge 0$. We recall that a function $u$ on $\Omega$ is called foliated Schwarz symmetric with respect to some unit vector $e \in \R^N$ is $u$ is axially symmetric with respect to the axis $\R e$ and nonincreasing in the angle $\theta = \arccos x \cdot e$. In the case where $p \ge 3$, the same symmetry is shared more generally by every solution $u$ of (\ref{1.4}) with Morse index $\mu(u) \le N$, the space dimension, see \cite{Pacella-Weth}.

In the case of the system~(\ref{eq:henon-1-zh}), we need to assume cooperativity again to recover foliated Schwarz symmetry. More precisely, if $p \ge 3$ and $\partial_{uv}F(u,v)>0$ for $u,v>0$, then every solution $(u,v)$ with Morse index $\mu(u,v) \le N$  is foliated Schwarz symmetric, i.e., both components $u$ and $v$ are foliated Schwarz symmetric with respect to the {\em same} unit vector $e$, see \cite[Theorem 1.4]{D-P}. Such a property is not expected in the non-cooperative case. For a study of symmetry properties in this case, we refer the reader to the recent papers \cite{T-W, S-T}.

We also mention related work on symmetry of solutions to the related second order Hamiltonian PDE system
 \begin{equation}\label{1.3}
 \left\{\begin{aligned}
 &-\Delta u=|x|^\beta |v|^{q-1}v&~ \hbox{in} ~&~ \Omega,\\
&-\Delta v=|x|^\alpha |u|^{p-1}u&~\hbox{in}~ &~ \Omega,\\
&u=v=0 &~~\hbox{on} ~&~\partial\Omega 
\end{aligned}\right.
\end{equation}
in the unit ball $\Omega \subset \mathbb{R}^N$, where $~\alpha, \beta> 0,~p,q>0$ and $\frac{1}{p+1}+\frac{1}{q+1}>\frac{N-2}{N}$.
In \cite{C-R}, Calanchi and Ruf have introduced the notion of ground state solutions, and they show symmetry breaking of these solutions for large values of $\alpha$ or $\beta$. 
Moreover, in \cite{B-S-R-1,B-S-R-2},  Bonheure, dos Santos and Ramos proved that ground
state solutions always exhibit foliated
Schwarz symmetry, and they present precise conditions on the parameters under which
the ground state solutions are not radially symmetric. Since there is no straightforward notion of Morse index of solutions of (\ref{1.3}), Theorem~\ref{th1.4} does not seem to have an analogue for (\ref{1.3}). Instead, the available results on symmetry breaking of ground state solutions of (\ref{1.3}) rely on a direct energy comparison involving radial and nonradial test functions.

We close the introduction with a brief outline of the strategy of the proof of Theorem~\ref{th1.4} and the structure of this paper. The main argument is given in Section~\ref{sec:symm-break-henon}. Here we argue by contradiction, assuming that there exists a sequence of numbers $\alpha_k>0$ with $\alpha_k \to \infty$ for $k \to \infty$ and, for every $k$, a nontrivial radial solution $(\tilde u_k,\tilde v_k)$ of (\ref{eq:henon-1-zh}) with $\alpha = \alpha_k$ and such that the Morse index of $(\tilde u_k,\tilde v_k)$ remains finite as $k \to \infty$. A spectral analysis using spherical harmonics then implies that an associated weighted radial eigenvalue problems only admits nonnegative eigenvalues. Inspired by Byeon and Wang \cite{B-W}, we then use a change of variable $r=|x|= e^{-\beta_k t}$ with $\beta_k = \frac{N}{N+\alpha_k}$ to transform both the $k$-dependent nonlinear system and the associated eigenvalue problem to the half line. Moreover, using the stability information derived in the previous step, we deduce local a priori bounds on the transformed sequence of solutions via a contradiction argument based on a blow up analysis. The a priori bounds then allow to pass to the limit along a subsequence and to deduce the existence of a stable solution of an associated limit problem either on $\R$ or on the half line. In Section~\ref{sec:preliminaries}, we derive a corresponding Liouville theorem which excludes the existence of stable solutions of these limit problems, and this yields a contradiction. Some technical parts of the argument, in particular regarding a variant of the very useful doubling lemma of Polacik, Quittner and Souplet \cite{PQS}, are postponed to the appendix of the paper.

\section{Preliminary results}
\label{sec:preliminaries}
In the present section, we collect some preliminary results which will be used in the proof of Theorem~\ref{th1.4}. Throughout this section, we assume that the function $F: \R^2 \to \R$ satisfies assumption (F) from the introduction with some $p>2$. We start by noting some immediate consequences of (F). First, it follows that
\begin{equation}
  \label{eq:consequence-F-0}
F(u,v) \ge c_F (|u|^p + |v|^p) \qquad \text{for $u,v \in \R$}
\end{equation}
with $c_F:= \min \{F(u,v)\::\: u,v \in \R,\:|u|^p + |v|^p=1\}>0.$ Moreover, by differentiating the function $t \mapsto F(tu,tv)$ at $t=1$, we see that
\begin{equation}
\label{F-consequence-1}
p F(u,v) = \partial_u F(u,v)u  + \partial_v F(u,v) v.
\end{equation}
Next, it is easy to see that the partial derivatives $\partial_u F(u,v)$, $\partial_v F(u,v)$ are \mbox{$p-1$-homogeneous.} Consequently, by differentiating the function
$$
t \mapsto \partial_u F(tu,tv)u  + \partial_v F(tu,tv) v
$$
at $t=1$, we see that
\begin{equation}
\label{F-consequence-2}
\Bigl \langle D^2 F(u,v) {u \choose v},{u \choose v} \Bigr \rangle = (p-1) \bigl(\partial_u F(u,v)u  + \partial_v F(u,v) v\bigr)
\end{equation}
where
$$
D^2 F(u,v) = \left (
  \begin{array}{cc}
\partial_{uu} F(u,v) &  \partial_{uv}F(u,v)\\
\partial_{uv}F(u,v) &  \partial_{vv}F(u,v)
  \end{array}
\right)
$$
Here and in the following, we let $\langle \cdot, \cdot \rangle$ denote the inner product in $\R^2$. Combining these assumptions, we see in particular that
\begin{align}
\label{F-consequence-3}
\Bigl \langle D^2 F(u,v) {u \choose v},{u \choose v} \Bigr \rangle &- \partial_u F(u,v)u  - \partial_v F(u,v) v\\
 &\ge p(p-2) c_F (|u|^p+|v|^p)\qquad \text{for $u,v \in \R$.}\nonumber
\end{align}

In Section~\ref{sec:symm-break-henon}, we will study radial solutions of (\ref{eq:henon-1-zh}) after a transformation. This approach will lead us to consider ODE systems both on $\R$ and on the half line $[0,\infty)$. In the remainder of this section, we will be concerned with observations related to functions on $\R$ and on $[0,\infty)$. We start with an elementary estimate for $C^1$-functions on the half line.

\begin{lemma}
\label{approx-lemma}
Let $u \in C^1([0,\infty))$ be a function with
\begin{equation}
\label{assumption-approx-lemma}
u(0)=0\qquad \text{and}\qquad   \int_0^\infty [u']^2\,dt < \infty.
\end{equation}
Then we have
\begin{equation}
  \label{eq:approx-lemma-limit-0}
u^2(t) \le t \int_0^\infty [u']^2\,d\tau \qquad \text{for all $\:t \ge 0$,}
\end{equation}
and there exists a sequence
$(\psi_n)_n$ in $C^\infty_c(0,\infty)$ with
\begin{equation}
  \label{eq:approx-lemma-limit}
\lim_{n \to \infty} \int_0^\infty (u-\psi_n)'^2\,dt = 0 \qquad \text{and}\qquad e^{-\delta t}(u-\psi_n)(t) \to 0
\end{equation}
uniformly in $t \in [0,\infty)$ as $n \to \infty$ for any $\delta>0$.
\end{lemma}

\begin{proof}
By (\ref{assumption-approx-lemma}) and H\"older's inequality we have
$$
|u(t)| = \Bigl| \int_0^t u'(\tau)\,d\tau\Bigr| \le  \sqrt{C_u t} \quad \text{for $t \ge 0$}\qquad \text{with}\qquad C_u:=\int_0^\infty [u']^2\,d\tau,
$$
as claimed in (\ref{eq:approx-lemma-limit-0}). Let $\phi \in C_c^{\infty}(\R)$ with $0 \le \phi \le 1$ and $\phi \equiv 1$ on $[-1,1]$, $\phi \equiv 0$ on $\R \setminus [-2,2]$. Moreover, let
$$
\phi_n \in C_c^\infty(0,\infty),\qquad \phi_n(t)= \phi(\frac{\ln t}{n})
$$
We then note that $0 \le \phi_n \le 1$ on $(0,\infty)$ and
$$
\phi_n(t) \to 1 \qquad \text{as $n \to \infty\quad $ for $t \in (0,\infty)$,}
$$
where the convergence is uniform in compact subsets of $(0,\infty)$. Let $\psi_n := \phi_n u \in C^1_c(0,\infty)$ for $n \in \N$. Then we have
$$
e^{-\delta t}|u(t)-\psi_n(t)|= e^{-\delta t}|(1-\phi_n(t))u(t)|\le \sqrt{C_u t}e^{-\delta t}|1-\phi_n| \qquad \text{for $n \in \N$, $t\ge 0$,}
$$
where the RHS tends to zero as $n \to \infty$ uniformly in $t \ge 0$. It thus remains to show the first limit in (\ref{eq:approx-lemma-limit}). For this we note that
\begin{align*}
\int_0^\infty [(u-\psi_n)']^2\,dt &= \int_0^\infty [(1-\phi_n)u' - \phi_n' u]^2\,dt \\
&\le 2 \int_0^\infty \Bigl(
(1-\phi_n)^2[u']^2 + [\phi_n']^2 u ^2\Bigr)\,dt,
\end{align*}
where, by Lebesgue's theorem,
$$
\int_0^\infty (1-\phi_n)^2[u']^2 \to 0 \qquad \text{as $n \to \infty$,}
$$
and
\begin{align*}
&\int_0^\infty [\phi_n']^2 u^2\,dt = \frac{1}{n^2} \int_0^\infty  \frac{1}{t^2}[\phi']^2(\frac{\ln t}{n})u^2(t)\,dt \le \frac{4 C_u}{n^2} \int_0^\infty \frac{1}{t}[\phi']^2(\frac{\ln t}{n})\,dt \\
&=\frac{4 C_u}{n^2} \int_\R  [\phi']^2(\frac{s}{n}) \,ds =\frac{4 C_u}{n^2} \int_{-2n}^{2n}  [\phi']^2(\frac{s}{n}) \,ds
\le \frac{16 C_u \|\phi'\|_\infty^2}{n} \to 0 \quad \text{as $n \to \infty$.}
\end{align*}
The proof is thus finished.
\end{proof}

Next, we state a Liouville Theorem for bounded solutions $(u,v)$ of the ODE system
\begin{equation}
  \label{eq:limit-system-rescaling-without-bc}
 \left \{
 \begin{aligned}
-u''   &= \partial_u F(u,v)&&\qquad \text{in $I$,}\\
-v''   &= \partial_v F(u,v)&&\qquad \text{in $I$,}
   \end{aligned}
\right.
\end{equation}
where
\begin{equation}
\label{sec:limit-system-1-0}
I = \R \qquad \text{or}\qquad I = (0,\infty).
\end{equation}
We need to introduce some notation. Let $(u,v) \in L^\infty(I,\R^2) \cap C^2(I,\R^2)$ be a fixed solution of (\ref{eq:limit-system-rescaling-without-bc}). We consider the quadratic form
$q_{u,v}$ on $C^1_c(I,\R^2)$ defined by
$$
q_{u,v}(\phi) := \int_{I} \bigl(|\phi'|^2 - \bigl \langle D^2 F(u,v) \phi, \phi  \bigr \rangle
\bigr)\,dt
$$
If $\Omega \subset I$ is an open subset, we say that $(u,v)$ is {\em stable in $\Omega$} if $q_{u,v}(\phi,\psi) \ge 0$ for all $\phi,\psi \in C^1_c(\Omega).$ Moreover, we say that $(u,v)$ is stable outside a compact set if $(u,v)$ is stable in $I \setminus K$ for some compact set $K \subset I$. We then have the following nonexistence result.

\begin{theorem}
\label{sec:limit-system-1}
Let $I$ satisfy (\ref{sec:limit-system-1-0}), and let $(u,v) \in L^\infty(I,\R^2) \cap C^2(I,\R^2)$ be a solution of (\ref{eq:limit-system-rescaling-without-bc}) which is stable outside a compact set.\\
Then $u  \equiv v \equiv 0$.
\end{theorem}

\begin{remark}
For a solution $(u,v) \in \bigl[L^\infty(I) \cap C^2(I)\bigr]^2$ of (\ref{eq:limit-system-rescaling-without-bc}), one may define the Morse index as the maximal $k \in \N \cup \{0,\infty)$ such that there exists a $k$-dimensional subspace $X \subset C^1_c(I)\times C_c^1(I)$ with the property that
$$
q_{u,v}(\phi,\psi) < 0 \qquad \text{for every $(\phi,\psi) \in X  \setminus \{0\}$.}
$$
A standard and straightforward argument shows that a solution
$(u,v) \in \bigl[L^\infty(I) \cap C^2(I)\bigr]^2$ of (\ref{eq:limit-system-rescaling-without-bc}) with finite Morse index is stable outside a suitable compact set $K \subset I$. Therefore the conclusion of Theorem~\ref{sec:limit-system-1} also applies to solutions with finite Morse index. In the present paper, we apply Theorem~\ref{sec:limit-system-1} only to the case where $(u,v)$ is stable, i.e., where $q_{u,v}(\phi,\psi) \ge 0$ for all $\phi,\psi \in C^1_c(I).$
\end{remark}

For the proof of Theorem~\ref{sec:limit-system-1}, we first need the following observation.

\begin{lemma}
\label{sec:limit-system-2}
Let $I$ satisfy (\ref{sec:limit-system-1-0}), and let $(u,v) \in L^\infty(I,\R^2) \cap C^2(I,\R^2)$ be a solution of (\ref{eq:limit-system-rescaling-without-bc}) with $(u,v) \not \equiv (0,0)$. Then there exists $\eps,\delta>0$ and a sequence $(r_n)_n \subset I$ with $r_n \to \infty$ as $n \to \infty$ and
\begin{equation}
  \label{eq:r-n-eps-delta-est}
\int_{r_n-\eps}^{r_n+\eps}(|u|^p+|v|^p)\,dt \ge \delta.
\end{equation}
\end{lemma}

\begin{proof}
By (\ref{eq:limit-system-rescaling-without-bc}), we see that the function
$$
E:= \frac{1}{2}\Bigl(u'^2+v'^2\Bigr) + F(u,v)
$$
is constant on $I$. Let $c_{u,v}$ denote the constant value of this function. Then $c_{u,v}>0$ by assumption (F) and since $(u,v) \not \equiv (0,0)$. Since $u^2+v^2$ is a bounded function on $I$, there exists a sequence $(t_n)_n \subset I$ with $t_{n+1} \ge t_n +1$ for every $n \in \N$ and
$$
u'(t_n)u(t_n)+ v'(t_n)v(t_n)= \frac{d}{dt}\Big|_{t_n}\Bigl(u^2+v^2\Bigr) \to 0 \qquad \text{as $n \to \infty$.}
$$
Suppose by contradiction that
\begin{equation}
  \label{eq:contradiction-decay}
u(t),v(t) \to 0 \qquad \text{as $t \to \infty$.}
\end{equation}
Multiplying (\ref{eq:limit-system-rescaling-without-bc}) with $u$, $v$ respectively and integrating by parts over $(t_n,t_{n+1})$,  we then find that
$$
\int_{t_n}^{t_{n+1}}\Bigl(u'^2+v'^2\Bigr)\,dt= o(1) + \int_
{t_n}^{t_{n+1}}F(u,v)\,dt = o(|t_{n+1}-t_n|)
$$
as $n \to \infty$. Thus there exists numbers $s_n \in (t_n,t_{n+1})$, $n \in \N$ with
$$
u'(s_n)^2 + v'(s_n)^2 \to 0 \qquad \text{as $n \to \infty$.}
$$
Moreover, $u(s_n),v(s_n) \to 0$ as $n \to \infty$ by (\ref{eq:contradiction-decay}). By definition of $E$, this contradicts the fact that $c_{u,v}>0$. Hence (\ref{eq:contradiction-decay}) is false, and so there exists $\delta>0$ and a sequence $(r_n)_n \subset I$ with $r_n \to \infty$ and
$$
|u|^p(r_n) + |v|^p(r_n) \ge 2 \delta \qquad \text{for all $n \in \N$.}
$$
Moreover, since $u'', v'' \in L^\infty(I)$ as a consequence of (\ref{eq:limit-system-rescaling-without-bc}), we may choose $\eps>0$ sufficiently close such that
$$
\int_{r_n- \eps}^{r_n+\eps}(|u|^p+|v|^p)\,dx \ge \delta \qquad \text{for all $n \in \N$,}
$$
as claimed.
\end{proof}

\begin{proof}[Proof of Theorem~\ref{sec:limit-system-1}]
We suppose by contradiction that $(u,v) \not \equiv (0,0)$. Let $\psi \in C^1_c(I \setminus K)$. The stability assumption applied to $(u \psi, v \psi) \in C^1_c(I \setminus K,\R^2)$ then yields
\begin{align}
&0 \le  \int_{I} \Bigl([(\psi u)']^2 + [(\psi v)']^2
- \Bigl \langle D^2 F(u,v)  {\psi u \choose \psi v}, {\psi u \choose \psi v} \Bigr) \rangle 
\Bigr)\,dt \nonumber\\
&= \int_{I} \Bigl( u'(\psi^2 u)' +  v' (\psi^2 v)' +(u^2+v^2) [\psi']^2 
- \Bigl \langle D^2 F(u,v)  {\psi u \choose \psi v}, {\psi u \choose \psi v} \Bigr \rangle
\Bigr)dt\nonumber\\
&= \int_{I} \Bigl([-u''- \partial_u F(u,v)](\psi^2 u)  +
[-v''- \partial_v F(u,v)](\psi^2 v)+ (u^2+v^2) [\psi']^2\Bigr)dt\nonumber\\
&\qquad -\int_I \psi^2 \Bigl(\,\Bigl\langle D^2 F(u,v) {u \choose v}, {u \choose v} \Bigr \rangle - \partial_u F(u,v)u- \partial_v F(u,v)v\Bigr)\,dt\nonumber\\
&\le \int_I (u^2+v^2) [\psi']^2\,dt - p (p-2)c_F \int_I  (|u|^p   + |v|^p)\psi^2 \,dt,  \label{eq:psi-est}
\end{align}
where in the last step (\ref{F-consequence-3}) and (\ref{eq:limit-system-rescaling-without-bc}) were used.
Now, for arbitrary $\phi \in C^1_c(I \setminus K)$, we apply (\ref{eq:psi-est}) to $\psi= \phi^{\frac{p}{p-2}}$ to get
$$
\int_I \Bigl(|u|^p + |v|^p\Bigr) \phi^{\frac{2p}{p-2}}\,dt \le \frac{p}{c_F(p-2)^3} \int_I (u^2+v^2)\phi^{\frac{4}{p-2}}[\phi']^2\,dt\\
$$
Combining this with Young's inequality yields, for $\tau>0$,
\begin{align*}
\int_I &\bigl(|u|^p + |v|^p\bigr) \phi^{\frac{2p}{p-2}}\,dt\\
& \le \frac{2\,\tau^{\frac{p}{2}} }{c_F(p-2)^3}\int_I \bigl(|u|^p + |v|^p\bigr) \phi^{\frac{2p}{p-2}}\,dt + \frac{2}{c_F(p-2)^2 \tau^{\frac{2p}{p-2}}}\int_{I} |\phi'|^{\frac{2p}{p-2}}\,dt.
\end{align*}
Choosing $\tau>0$ such that $\frac{2 \,\tau^{\frac{p}{2}}}{c_F(p-2)^3}= \frac{1}{2}$, we thus conclude that
\begin{equation}
  \label{eq:phi-est}
\int_I \bigl(|u|^p + |v|^p\bigr) \phi^{\frac{2p}{p-2}}\,dt \le C_\tau \int_{I} |\phi'|^{\frac{2p}{p-2}}\,dt
\end{equation}
with $C_\tau := \frac{4 \tau^{\frac{2p}{2-p}}}{c_F(p-2)^2}$. Next, let $\phi_0 \in C^1_c(\R)$ satisfy
$$
0 \le \phi_0 \le 1, \quad \text{$\phi_0 \equiv 1$ on $[-1,1]$}\quad \text{and}\quad \text{$\phi_0 \equiv 0$ on $\R \setminus [-2,2]$.}
$$
For $\rho>0$ and $r \in \R$, we then consider
$$
\phi_{\rho,r} \in C^1_c(\R),\qquad \phi_{\rho,r}(t)= \phi(\rho(t-r)),
$$
so that
$$
\int_\R [\phi_{\rho,r}'(t)]^{\frac{2p}{p-2}}\,dt = \rho^{\frac{2p}{p-2}} \int_{\R} [\phi'(\rho(t-r))]^{\frac{2p}{p-2}}\,dt =
\rho^{\frac{p+2}{p-2}} \int_{\R}\phi_0\,dt.
$$
With $\eps,\delta>0$ given by Lemma~\ref{sec:limit-system-2}, we may now fix $\rho<\frac{1}{\eps}$ sufficiently small such that
\begin{equation}
  \label{eq:delta-a-min-est}
\int_\R [\phi_{\rho,r}'(t)]^{\frac{2p}{p-2}}\,dt < \frac{\delta}{C_\tau} \qquad \text{for all $r \in \R$.}
\end{equation}
Since $r_n \to \infty$ for the sequence $(r_n)_n$ given by Lemma~\ref{sec:limit-system-2}, there exists $n_0 \in \N$ such that
$$
\supp \, \phi_{\rho,r_n} \subset I \setminus K \qquad \text{for $n \ge n_0$.}
$$
Since moreover $\phi_{\rho,r_n} \equiv 1$ on $[r_n-\eps,r_n+\eps]$ by our choice of $\rho$, we can use (\ref{eq:phi-est}) and (\ref{eq:delta-a-min-est}) to estimate that
$$
\int_{r_n-\eps}^{r_n+\eps}(|u|^p+|v|^p)\,dt \le \int_I (|u|^p+|v|^p)\phi_{\rho,r_n}^{\frac{2p}{p-2}}
\,dt \le C_\tau \int_\R [\phi_{\rho,r_n}'(t)]^{\frac{2p}{p-2}}\,dt
< \delta
$$
for $n \ge n_0$, contrary to (\ref{eq:r-n-eps-delta-est}). The contradiction shows that $u \equiv v \equiv 0$, as claimed.
\end{proof}

We close this section with estimates for a more general ODE system which arises when studying radial solutions of
(\ref{eq:henon-1-zh}) after a transformation. These estimates will be used in Proposition~\ref{sec:morse-index-system} below.

\begin{lemma}
\label{pohozaev-transformed}
Let $\gamma >0$, $\rho \ge 0$, $N \ge \frac{p\rho}{2} + \frac{(p-2)\gamma}{2}$ and $\nu_1, \nu_2 \ge 0$ be constants. Moreover, let
$(u,  v) \in C^2([0,\infty),\R^2)$ be a bounded solution of the system
\begin{equation}
  \label{eq:henon-transformed-0}
\left \{
 \begin{aligned}
-\bigl(e^{-\gamma t} u'\bigr)' + \nu_1 e^{-\rho t} u  &= e^{-N t}\partial_u F(u,v)&&\quad \text{in $(0,\infty)$,}\\
-\bigl(e^{-\gamma t} v'\bigr)' + \nu_2 e^{-\rho t} v  &= e^{-N t}\partial_v F(u,v)&&\quad \text{in $(0,\infty)$,}\\
u (0)&=v(0)=0.
    \end{aligned}
 \right.
\end{equation}
Then we have
\begin{equation}
\label{pohozaev-transformed-1}
[u'(0)]^2+ [v'(0)]^2 \ge \frac{2(N+\gamma)}{p} \int_0^\infty \Bigl([(e^{-\gamma t}u)'\,]^2 + [(e^{-\gamma t}v)'\,]^2\Bigr)\,dt.
\end{equation}
Moreover, if $\gamma \le \frac{N}{3p}$, then there exists a constant $C>0$ depending only on $F,p$ and $N$ such that
\begin{equation}
\label{pohozaev-transformed-lower-bound}
[u'(0)]^2+ [v'(0)]^2 \ge C \qquad \text{if $(u,v) \not \equiv (0,0)$.}
\end{equation}
\end{lemma}

\begin{proof}
By (\ref{eq:henon-transformed-0}), we have that
\begin{align*}
&\frac{1}{2}\Bigl([u'(0)]^2+ [v'(0)]^2\Bigr) = -\frac{1}{2}\int_0^\infty \partial_t \Bigl([e^{-\gamma t} u']^2+[e^{-\gamma t}v']^2\Bigr) \,dt\\
&= -\int_0^\infty \Bigl(e^{-\gamma t}u' (e^{-\gamma t} u')'+ e^{-\gamma t}v'
(e^{-\gamma t}v')\Bigr)\,dt\\
&= \!\int_0^\infty\!\!\! e^{-\gamma t}\Bigl( e^{-Nt}\bigl(\partial_u F(u,v)u' + \partial_v F(u,v)v'\bigr)-e^{- \rho t}(\nu_1 uu'+\nu_2 vv')\!\Bigr)dt\\
&= \int_0^\infty \Bigl[e^{-(N+\gamma)t} [F(u,v)]'-\frac{e^{-(\rho+ \gamma)t}}{2}(\nu_1 u^2+\nu_2 v^2)'\Bigr]\,dt\\
&= (N+\gamma) \int_0^\infty e^{-(N+\gamma)t} F(u,v)\,dt -  \frac{\rho+ \gamma}{2} \int e^{-(\rho+ \gamma)t} (\nu_1 u^2+\nu_2 v^2) \,dt\\
& \ge (N+\gamma) \int_0^\infty e^{-\gamma t}\Bigl[e^{-Nt}F(u,v)- \frac{e^{-\rho t}}{p}(\nu_1 u^2+\nu_2 v^2)\Bigr]\,dt,
\end{align*}
where in the last step we used the assumption $N \ge \frac{p\rho}{2} + \frac{(p-2)\gamma}{2}$. Multiplying (\ref{eq:henon-transformed-0}) by $u$ resp. $v$ and integrating, we thus find that
\begin{align}
&\int_0^\infty \Bigl([(e^{-\gamma t}u)'\,]^2 + [(e^{-\gamma t}v)'\,]^2\Bigr)\,dt \nonumber \\
&=\int_0^\infty e^{-\gamma t}\bigl[\bigl(e^{-\gamma t} u'\bigr)'u + \bigl(e^{-\gamma t} v'\bigr)'v\bigr]\,dt \nonumber \\
&=\int_0^\infty e^{-\gamma t}\Bigl[e^{-Nt}(\partial_u F(u,v)u + \partial_v F(u,v)v)- e^{-\rho t}(\nu_1 u^2+\nu_2 v^2)\Bigr]\,dt \nonumber\\
&=\int_0^\infty e^{-\gamma t}\Bigl[p e^{-Nt} F(u,v)- e^{-\rho t}(\nu_1 u^2+\nu_2 v^2)\Bigr]\,dt \label{compactness-add-ref}\\
&\le \frac{p}{2(N+\gamma)} \Bigl([u'(0)]^2+ [v'(0)]^2\Bigr), \nonumber
\end{align}
as claimed in (\ref{pohozaev-transformed-1}). Here we used (\ref{F-consequence-1}). Moreover, for $t \ge 0$ we have
$$
e^{- \gamma t} u(t)
=  \int_{0}^t (e^{-\gamma s} u)'\,ds \le \sqrt{t} \Bigl(\int_0^\infty [(e^{-\gamma s} u)'\,]^2\,ds\Bigr)^{\frac{1}{2}}
$$
and thus, if $\gamma \le \frac{N}{3p}$,
\begin{align*}
&e^{- \frac{4N}{3p}t} u^2(t) =   e^{-(\frac{4N}{3p}+2\gamma)t}[e^{- \gamma t} u(t)]^2 \le
te^{-(\frac{4N}{3p}+2\gamma)t} \int_0^\infty [(e^{-\gamma s} u)'\,]^2\,ds\\
&\le t e^{-\frac{2N}{3p}t} \int_0^\infty [(e^{-\gamma s} u)\,]'^2\,ds \le C_{N,p} \int_0^\infty [(e^{-\gamma s} u)'\,]^2\,ds
\end{align*}
with $C_{N,p}:= \max \limits_{t \ge 0}\bigl(t e^{-\frac{2N}{3p}t}\bigr)$. The same estimate holds for $v$, and then by (\ref{compactness-add-ref}) we get
\begin{align}
\int_0^\infty &\Bigl([(e^{-\gamma t}u)'\,]^2 + [(e^{-\gamma t}v)'\,]^2\Bigr)\,dt \le p \int_0^\infty e^{-(N+\gamma) t}F(u,v)\,dt
\nonumber\\
&\le p C_F \int_0^\infty e^{- N t} \Bigl(|u(t)|^2+|v(t)|^2\Bigr)^{\frac{p}{2}}\,dt \nonumber \\
&= p C_F \int_0^\infty e^{- \frac{N}{3} t} \Bigl(e^{-\frac{4N}{3p} t} |u(t)|^2+e^{-\frac{4N}{3p} t}|v(t)|^2\Bigr)^{\frac{p}{2}}\,dt \nonumber\\
&\le p C_F {C_{N,p}}^{\frac{p}{2}} \int_{0}^\infty e^{-\frac{N}{3} t}\,dt \Bigl[\int_0^\infty \Bigl([(e^{-\gamma t}u)'\,]^2 + [(e^{-\gamma t}v)'\,]^2\Bigr)\,dt \Bigr]^{\frac{p}{2}} \nonumber\\
&=\frac{3 p C_F {C_{N,p}}^{\frac{p}{2}}}{N}\Bigl[\int_0^\infty \Bigl([(e^{-\gamma t}u)'\,]^2 + [(e^{-\gamma t}v)'\,]^2\Bigr)\,dt \Bigr]^{\frac{p}{2}}.\nonumber
\end{align}
So if $(u,v) \not \equiv (0,0)$ we have
$$
\int_0^\infty \Bigl([(e^{-\gamma t}u)'\,]^2 + [(e^{-\gamma t}v)'\,]^2\Bigr)\,dt  \ge \Bigl(\frac{N}{3 p C_F {C_{N,p}}^{\frac{p}{2}}}\Bigr)^{\frac{2}{p-2}}.
$$
Combining this with (\ref{pohozaev-transformed-1}), we thus conclude that
$$
[u'(0)]^2+ [v'(0)]^2 \ge \frac{2 N}{p} \int_0^\infty \Bigl([(e^{-\gamma t}u)'\,]^2 + [(e^{-\gamma t}v)'\,]^2\Bigr)\,dt  \ge \frac{2 N}{p}\Bigl(\frac{N}{3 p C_F {C_{N,p}}^{\frac{p}{2}}}\Bigr)^{\frac{2}{p-2}},
$$
as claimed in (\ref{pohozaev-transformed-lower-bound}).
\end{proof}

\section{Symmetry breaking for the H{\'e}non-Schr\"{o}dinger system}
\label{sec:symm-break-henon}

The present section is completely devoted to the proof of Theorem~\ref{th1.4}. As noted in the introduction, Part (i) is a direct consequence of Part (ii), so it only remains to prove Part (ii).

Arguing by contradiction, we suppose that there exists $m>0$, a sequence of numbers $\alpha_k>0$ with $\alpha_k \to \infty$ for $k \to \infty$ and, for every $k$, a nontrivial radial solution $(\tilde u_k,\tilde v_k)$ of
\begin{equation}\label{eq:henon-1}
\left \{
 \begin{aligned}
 -\Delta \tilde u_k + \mu_1 \tilde u_k  &= |x|^{\alpha_k}\partial_u F(\tilde u_k,\tilde v_k)&&\qquad \text{in $\Omega$,}\\
 -\Delta \tilde v_k + \mu_2 \tilde v_k  &= |x|^{\alpha_k}\partial_v F(\tilde u_k,\tilde v_k)&&\qquad \text{in $\Omega$,}\\
u&=v=0&&\qquad \text{on $\partial \Omega$,}
    \end{aligned}
 \right.
\end{equation}
with $\alpha=\alpha_k$ such that
\begin{equation}
\label{eq:morse-upper-bound}
\mu(\tilde u_k,\tilde v_k) \le m \qquad \text{for all $k \in \N$.}
\end{equation}
Let $L_k:= L_{\tilde u_k,\tilde v_k}^{\alpha_k}$ denote the linearized operator at
$(\tilde u_k,\tilde v_k)$ as defined in (\ref{eq:linearized-operator}), i.e.,
$$
L_k \phi := -\Delta  { \phi_1 \choose \phi_2 } +
{\mu_1 \phi_1 \choose \mu_2 \phi_2 } -  |x|^{\alpha_k} D^2 F(\tilde u_k,\tilde v_k) \phi  .
$$
By (\ref{eq:morse-upper-bound}), $L_k$ has most $m$ negative Dirichlet eigenvalues. Let $\Delta_\theta$ denote
the Laplace-Beltrami-Operator on the unit sphere. In the following, we restrict our attention to eigenfunctions of the form
$$
x \mapsto \phi(x) = Y_{l}(\theta)w(r)\quad
\text{with}\quad  w(r)=  {w_1(r) \choose w_2(r)}
$$
for  $r = |x|$ and $\theta = \frac{x}{|x|}$. Here $Y_l$ is a
spherical harmonic of degree $\ell$, i.e. an eigenfunction of $-\Delta_\theta$ on the
unit sphere $\mathbb S$ corresponding to the eigenvalue $\lambda_\ell := l(l+N-2)$.
Then the problem
$$
L_k \phi   = \mu  \phi  \quad
\text{in $\Omega$},\qquad \phi = 0 \quad \text{on
$\partial \Omega$}
$$
reduces to
\begin{equation}
  \label{eq:radial-variable-linear-problem}
-\Delta_r  w  +
\frac{\lambda_\ell}{r^2}w - V_k (r) w = \mu w \quad \text{on
$[0,1]$}, \qquad w(1)=0,
\end{equation}
with
$$
V_k(r):= r^{\alpha_k}D^2 F(\tilde u_k,\tilde v_k)(r)  - \left (\begin{array}{cc}
\mu_1 & 0 \\
0 & \mu_2
\end{array}
 \right)
\quad  \in \R^{2 \times 2}\qquad \text{for $r \in [0,1]$.}
$$
We call $\mu \in \R$ an eigenvalue of (\ref{eq:radial-variable-linear-problem}) if there exists a solution $w \in C^2([0,1])\setminus \{0\}$ of (\ref{eq:radial-variable-linear-problem}). We claim that
\begin{equation}
  \label{eq:eigenvalue-radial-text}
  \begin{aligned}
&\text{for every $k \in \N$ there exists $\ell_k \in \{0,\dots,m\}$ such that}\\
&\text{(\ref{eq:radial-variable-linear-problem}) admits only nonnegative eigenvalues for $\ell = \ell_k$.}
  \end{aligned}
\end{equation}
Indeed, if this is not the case, then there exists $k \in \N$ and $w_1,\dots,w_m \in C^2((0,1)) \cap C([0,1]) \setminus \{0\}$ such that $w_j$ solves (\ref{eq:radial-variable-linear-problem}) with some eigenvalue $\mu=\mu_j<0$ for $0=1,\dots,m$. We may then pick a spherical harmonic $Y_j$ of degree $j$ for $j=0,\dots,m$ and define $v_j \in H^2(\Omega) \cap H^1_0(\Omega)$ in polar coordinates by $v_j(r,\theta)=Y_j(\theta) w_j(r)$. Then the functions $v_0,\dots,v_m$ are orthogonal in $L^2(\Omega)$, since the functions $Y_0,\dots,Y_m$ are orthogonal in $L^2(\mathbb S)$. Moreover, every $v_j$ is an eigenfunction of $L_k$ associated with a negative eigenvalue, and thus $L_k$ has at least $m+1$ negative eigenvalues. This contradicts (\ref{eq:morse-upper-bound}), and thus the claim~(\ref{eq:eigenvalue-radial-text}) is true.\\
As a consequence of (\ref{eq:eigenvalue-radial-text}), there exists $\ell \in \{0,\dots,m\}$ such that, after passing to a subsequence in $k$,
\begin{equation}
  \label{eq:eigenvalue-radial-text-1}
  \begin{aligned}
&\text{the eigenvalue problem (\ref{eq:radial-variable-linear-problem}) admits only}\\
&\text{nonnegative eigenvalues for every $k \in \N$.}
  \end{aligned}
\end{equation}
It is now convenient to use, inspired by Byeon-Wang \cite{B-W}, the change of variable
$r= e^{-\beta_k t}$ with $\beta_k = \frac{N}{N+\alpha_k}$, which leads to
$\partial_r = \frac{e^{\beta_k t}}{\beta_k} \partial_t$ and therefore
$$
\Delta_r = r^{1-N}\partial_r (r^{N-1} \partial_r) = \frac{e^{\beta_k N
t}}{\beta_k^{2}} \partial_t \Bigl(e^{(2-N) \beta_k t}\partial_t\Bigr).
$$
Hence, setting $\gamma_k:= (N-2) \beta_k$,
we see that the transformed functions
$$
u_k, v_k: [0,\infty) \to \R, \qquad u_k(t):=    \beta_k^{\frac{2}{p-2}} \tilde u_k (e^{-\beta_k t}), \quad v_k(t):=    \beta_k^{\frac{2}{p-2}} \tilde v_k (e^{-\beta_k
t})
$$
solve the system
\begin{equation}
  \label{eq:henon-transformed}
\left \{
 \begin{aligned}
-\bigl(e^{-\gamma_k t} {u_k}'\bigr)' + {\beta_k}^2\mu_1 e^{-{\beta_k} N t}{u_k}  &= e^{-N t}\partial_u F(u_k,v_k)&&\quad \text{in $(0,\infty)$,}\\
-\bigl(e^{-\gamma_k t} {v_k}'\bigr)' + {\beta_k}^2  \mu_2 e^{-{\beta_k} Nt}{v_k}  &= e^{-N t}\partial_v F(u_k,v_k)&&\quad \text{in $(0,\infty)$,}\\
{u_k} (0)&={v_k}(0)=0.
    \end{aligned}
 \right.
\end{equation}
Here and in the following, the prime stands for $\partial_t$. Moreover, setting
$$
h(t)= w(e^{-\beta_k t}) = {w_1(e^{-\beta_k t}) \choose
w_2(e^{-\beta_k t})}
$$
and putting $\lambda= \lambda_\ell \ge 0$, we see that (\ref{eq:radial-variable-linear-problem}) transforms into
\begin{equation}
  \label{eq:radial-variable-linear-problem-1}
-\bigl(e^{-\gamma_k t}h' \bigr)' +\beta_k^2 \lambda e^{-\gamma_k t} h- U_k(t) h = \mu \beta_k^2 e^{-\beta_k N t} h \quad \text{on
$[0,\infty)$}
\end{equation}
subject to the conditions
\begin{equation}
  \label{eq:radial-variable-linear-problem-1-cond}
h(0)=0,\qquad h(\infty)=\lim_{t \to \infty}h(t) \quad \text{exists,}
\end{equation}
where
\begin{align*}
U_k(t)&:=  \beta^2 e^{- N t} D^2 F(u_k,v_k)(e^{-\beta_k t}) - \beta_k^{2} e^{- \beta_k N  t} \left (\begin{array}{cc}
\mu_1 & 0 \\
0 & \mu_2
\end{array}
 \right)\\
&=  e^{- N t} D^2 F(u_k, v_k)(t) - \beta_k^{2} e^{- \beta_k N t} \left (\begin{array}{cc}
\mu_1 & 0 \\
0 & \mu_2
\end{array}
 \right)\;\in \R^{2 \times 2}\qquad \text{for $t \ge 0$.}
\end{align*}
Here we used the fact that the function $\R^2 \to \R,\; (u,v) \to D^2 F(u,v)$ is $(p-2)$-homogeneous, which follows easily from assumption $(F)$. We call $\mu \in \R$ an eigenvalue of (\ref{eq:radial-variable-linear-problem-1}),~(\ref{eq:radial-variable-linear-problem-1-cond}) if there exists a bounded solution
$h \in C^2([0,\infty)) \setminus \{0\}$ of (\ref{eq:radial-variable-linear-problem}) such that $~(\ref{eq:radial-variable-linear-problem-1-cond})$ holds. It then follows immediately from~(\ref{eq:eigenvalue-radial-text-1}) that the eigenvalue problem (\ref{eq:radial-variable-linear-problem-1}) admits only nonnegative eigenvalues for every $k \in \N$. Applying Lemma~\ref{sec:morse-index-system-help-lemma} from the Appendix for fixed $k$ with $\gamma= \gamma_k$, $\delta= N \beta_k$ and $U(t):= e^{\beta_k N t}U_k(t)$, we deduce that
\begin{equation}
  \label{eq:Q_k-nonnegative}
Q_k(\varphi) \ge 0 \qquad \text{for every $\varphi \in C_c^\infty((0,\infty),\R^2)$,}
\end{equation}
where
$$
Q_k(\varphi):= \int_{0}^\infty \Bigl(e^{-\gamma_k t} |\varphi'(t)|^2 + \lambda e^{-\gamma_k t}|\phi(t)|^2  - e^{-N t} \bigl \langle U_k(t)\varphi(t),\varphi(t) \bigr \rangle\Bigr)\,dt.
$$
Hence we may apply Proposition~\ref{sec:morse-index-system} below to deduce that $u_k \equiv v_k \equiv 0$ for $k$ sufficiently large. This is a contradiction.

Thus, the following Proposition completes the proof of Theorem~\ref{th1.4}.

\begin{proposition}
\label{sec:morse-index-system} Let $\lambda,\mu_1,\mu_2 \ge 0$ be constants, and let $\beta_k, \gamma_k >0$, $k \in \N$ with $\lim \limits_{k \to \infty} \beta_k = \lim \limits_{k \to \infty} \gamma_k= 0$. Moreover, for $k \in \N$,
let $(u_k,  v_k) \in C^2([0,\infty),\R)$ be solutions of the system
\begin{equation}
  \label{eq:henon-transformed-1}
\left \{
 \begin{aligned}
-\bigl(e^{-\gamma_k t}u_k' \bigr)' + \beta_k^2 e^{-\beta_k N t} \mu_1 u_k  &= e^{-N t}\partial_u F(u_k,v_k)&&\quad \text{in $(0,\infty)$,}\\
-\bigl(e^{-\gamma_k t}v_k' \bigr)' + \beta_k^2 e^{-\beta_k N t}\mu_2 v_k  &= e^{-N t}\partial_v F(u_k,v_k)&&\quad \text{in $(0,\infty)$,}\\
u_k (0)&=v_k(0)=0.
    \end{aligned}
 \right.
\end{equation}
Assume furthermore that
\begin{equation}
  \label{eq:stability-k}
Q_k(\varphi) \ge 0 \qquad \text{for every $\varphi \in C^1_c((0,\infty),\R^2)$, $k \in \N$,}
\end{equation}
where
$$
Q_k(\varphi):= \int_{0}^\infty \Bigl(e^{-\gamma_k t} |\varphi'(t)|^2 + \lambda \beta_k^2  e^{-\gamma_k t}|\varphi(t)|^2  - \bigl \langle U_k(t)\varphi(t),\varphi(t) \bigr \rangle\Bigr)\,dt
$$
and $U_k: [0, \infty) \to \R^{2 \times 2}$ is defined by
$$
U_k(t):=  e^{-Nt} D^2 F(u_k,v_k)(t) - \beta_k^{2}e^{-\beta_k N t}  \left (\begin{array}{cc}
\mu_1 & 0 \\
0 & \mu_2
\end{array}
 \right)
$$
Then $(u_k,v_k) \equiv 0$ for $k$ sufficiently large.
\end{proposition}

\begin{proof}
We may rewrite the system (\ref{eq:henon-transformed-1}) as
$$
\left \{
 \begin{aligned}
-u_k'' + \gamma_k u_k'   + {\beta_k}^2 e^{\gamma_k t} W_1(t) u_k  &= e^{(\gamma_k-N) t}\partial_u F(u_k,v_k)\\
-v_k'' + \gamma_k v_k'  + {\beta_k}^2 e^{\gamma_k t} W_2(t) v_k  &= e^{(\gamma_k-N) t}\partial_v F(u_k,v_k)\\
u_k (0)&=v_k(0)=0,
    \end{aligned}
 \right.
$$
with $W_i(t):= e^{-\beta_k N t}\mu_i$ for $i=1,2$ and $t \ge 0$. By a blow up argument based on the Liouville Theorem~\ref{sec:limit-system-1} and a variant of the doubling lemma of Polacik, Quittner and Souplet \cite{PQS}, we first wish to show that the sequence $(u_k,v_k)_k$ is bounded in $L^\infty_{loc}([0,\infty),\R^2)$. For this we consider the functions
$$
M_k:= \max \{ |u_k|^{\frac{p-2}{2}}, |v_k|^{\frac{p-2}{2}}\} : [0,\infty) \to \R, \qquad k \in \N.
$$
Arguing be contradiction, let us assume that there exists a bounded sequence $(s_k)$ in $[0,\infty)$ such that ${N_k}:= M_k(s_k) \to \infty$ as $k \to \infty$. Applying Lemma~\ref{doubling-simplified-further} from the Appendix to $X:=[0,\infty)$,
we find another bounded sequence $(t_k)_k$ in $[0,\infty)$
such that, for $k \in \N$,
$$
M_k(t_k) \ge {N_k} \qquad \text{and}\qquad M_k(t) \le 2 M_k(t_k)\quad \text{for
  $t \in B_{\sigma_k {N_k}}(t_k) \cap [0,\infty)$,}
$$
where $\sigma_k:= \frac{1}{M_k(t_k)}$ for $k \in \N$. Passing to a subsequence, we may assume that
$$
t_k \to t_0 \in [0,\infty) \qquad \text{as $k \to \infty$.}
$$
We now put $c_k:= \frac{t_k}{\sigma_k}$ for $k \in \N$ and distinguish two cases. \\[0.2cm]
{\bf Case 1:} $c_k \to \infty$ as $k \to \infty$.\\
In this case,  we consider the functions $\bar u_{k}, \bar
v_{k}$ on \mbox{$I_k:= [-c_k,\infty)$} given by
$$
\bar u_k(t):= {\sigma_k}^{\frac{2}{p-2}}  u(t_k+ {\sigma_k} t),\quad \bar v_k(t):= {\sigma_k}^{\frac{2}{p-2}}  v(t_k + {\sigma_k} t)\qquad \text{for $k \in \N$.}
$$
These functions solve
$$
\left \{
 \begin{aligned}
-\bar u_k'' +  {\sigma_k} \gamma_k
\bar u_k'   + {\sigma_k}^{2} {\beta_k}^2 W_1^k(t) \bar u_k &=  e^{[\gamma_k-N](t_k+ \sigma_k t) }\partial_u F(\bar u_k,\bar v_k)
\\
-\bar v_k'' + {\sigma_k} \gamma_k \bar v_k'  + {\sigma_k}^2{\beta_k}^2 W_2^k(t) \bar v_k  &= e^{[\gamma_k-N](t_k+ \sigma_k t) }\partial_v F(\bar u_k,\bar v_k)
    \end{aligned}
 \right.
$$
in $I_k$ with
$$
W_i^k(t):= e^{\gamma_k (t_k + \sigma_k t)} W_i (t_k + \sigma_k t)\qquad \text{for $i=1,2$, $k \in \N$ and $t \in I_k$.}
$$
Moreover, we have $\max \{|\bar u_k(0)|,|\bar v_k(0)|\}=1$ and
$$
\max \{|\bar u_k(t)|,|\bar v_k(t)|\} \le 2^{\frac{2}{p-2}} \qquad \text{in $(-{N_k} ,{N_k} ) \cap I_k$}
$$
Since the functions $W_i^k$ remain locally bounded as $k \to \infty$, we may apply elementary ODE regularity estimates in order to pass to a subsequence with
$$
\bar u_k \to \bar u ,\quad \bar v_k \to \bar v \qquad \text{locally
uniformly in $\R$}
$$
where $\max \{\bar u(0), \bar v(0)\}=1$ and $(\bar u, \bar v)$
satisfies the limit system
\begin{equation}
  \label{eq:limit-system-rescaling-entire}
 \left \{
 \begin{aligned}
-\bar u''   &= e^{-Nt_0}\partial_u F(\bar u,\bar v), &&\qquad \text{in $\R$}\\
-\bar v''   &= e^{-Nt_0}\partial_v F(\bar u,\bar v), &&\qquad \text{in $\R$}\\
   \end{aligned}
\right.
\end{equation}
In particular, $(\bar u,\bar v)$ is nontrivial. By Theorem~\ref{sec:limit-system-1} -- applied with $e^{-N t_0} F$ in place of $F$-- it then follows that $(\bar u,\bar v)$ is not stable in $\R$, so there exists $\varphi= (\varphi_1,\varphi_2) \in C^1_c(\R,\R^2)$ such that
\begin{equation}
  \label{eq:not-stable-limit-1}
\int_{\R} \Bigl(|\varphi'(t)|^2
-e^{-N t_0} \bigl \langle D^2 F(\bar u,\bar v) \varphi, \varphi  \bigr \rangle(t)
\Bigr)\,dt <0.
\end{equation}
Since
$$
\lim_{k \to \infty} (\bar u_k(t), \bar v_k(t))= (\bar u(t), \bar v(t))\qquad \text{and} \qquad
\lim_{k \to \infty} (t_k+ \sigma_k t ) = t_0
$$
uniformly in $t$ on the support of $\varphi$, we also have that
\begin{align*}
&\sigma_k^2 U_k(t_k + \sigma_k t)\\
&= e^{-N(t_k + \sigma_k t)} D^2 F(\bar u_k,\bar v_k)(t) - \sigma_k^2 \beta_k^{2}  \left (\begin{array}{cc}
W_1 (t_k + \sigma_k t) & 0 \\
0 & W_2 (t_k + \sigma_k t)
\end{array}
 \right)\\
&\longrightarrow\; e^{-N t_0} D^2 F(\bar u,\bar v)(t) \qquad \text{uniformly in $t$ on the support of $\varphi$.}
\end{align*}
Here we used the fact that $\sigma_k ,\beta_k \to 0$ as $k \to \infty$. Recalling also that $\gamma_k \to 0$ as $k \to \infty$, it then follows from (\ref{eq:not-stable-limit-1}) that
\begin{align*}
q_k:= \int_{I_k}&
e^{-\gamma_k (t_k + \sigma_k t)} |\varphi'(t)|^2\\
+& \sigma_k^2 \Bigl(\beta_k^2 \lambda e^{-\gamma_k (t_k + \sigma_k t)}|\varphi(t)|^2 -  \bigl \langle U_k(t_k + \sigma_k t) \varphi (t),\varphi (t) \bigr \rangle
\Bigr)\,dt <0
\end{align*}
for $k$ sufficiently large. Fixing $k$ with this property and large enough to guarantee that $I_k$ contains the support of $\varphi$, we may then write
$$
\varphi(t)= \psi(t_k + \sigma_k t) \qquad \text{with}\quad \psi= (\psi_1,\psi_2) \in C^1_c((0,\infty),\R^2).
$$
A change of variable then shows that
$$
q_k = \sigma_k \!\int_{I_k}\!\bigl(
e^{-\gamma_k \tau} |\psi'(\tau)|^2 + \lambda \beta_k^2 e^{-\gamma_k \tau}|\psi(\tau)|^2 -
\langle \psi(\tau),
U_k(\tau)\psi(\tau)\rangle
\bigr)d\tau = \sigma_k Q_k(\psi).
$$
Sincr $\sigma_k>0$, we conclude that $Q_k(\psi)= \frac{q_k}{\sigma_k}>0$, contradicting the assumption (\ref{eq:stability-k}).\\
{\bf Case 2:} For a subsequence, $c_k:= \frac{t_k}{\sigma_k} \to c \ge 0$ as $k \to \infty$.\\[0.2cm]
In this case  we have $t_0 = 0$, and we consider the functions
$$
\bar u_{k}, \bar
v_{k}: [0,\infty) \to \R,\qquad
\bar u_k(t):= {\sigma_k}  u({\sigma_k} t),\; \bar v_k(t):= {\sigma_k}  v({\sigma_k} t)
$$
for $k \in \N$. These functions solve
$$
\left \{
 \begin{aligned}
-\bar u_k'' +  {\sigma_k} \gamma_k
\bar u_k'   + {\sigma_k}^{2} {\beta_k}^2
W_1^k(t) \bar u_k &=  e^{(\gamma_k-N) \sigma_k t }\partial_u F(u_k, v_k)
\\
-\bar v_k'' + {\sigma_k} \gamma_k \bar v_k'  + {\sigma_k}^2 {\beta_k}^2 W_2^k(t)
  \bar v_k  &= e^{(\gamma_k-N) \sigma_k t }\partial_v F(u_k,v_k)
    \end{aligned}
 \right.
$$
with
$$
W_i^k(t):= e^{\gamma_k \sigma_k t} W_i (\sigma_k t)\qquad \text{for $i=1,2$, $k \in \N$ and $t \in I_k$.}
$$
in $[0,\infty)$. Moreover, $\max \{|\bar u_k(c_k)|,|\bar v_k(c_k)|\}=1$ and
$$
\max \{|\bar u_k(t)|,|\bar v_k(t)|\} \le 2^{\frac{2}{p-2}} \qquad \text{in $[0,c_k + {N_k} )$}
$$
Here we suppose that $k$ is sufficiently large so that ${N_k} \ge c_k$. Applying elementary ODE regularity theory, we may pass to a subsequence such that
$$
\bar u_k \to \bar u ,\quad \bar v_k \to \bar v \qquad \text{locally
uniformly in $[0,\infty)$}
$$
where $\max \{\bar u(c), \bar v(c)\}=1$ and $(\bar u, \bar v)$
satisfies the limit problem
\begin{equation}
  \label{eq:limit-system-rescaling-half}
 \left \{
 \begin{aligned}
-\bar u''   &= \partial_u F(u,v), &&\qquad \text{in $[0,\infty)$}\\
-\bar v''   &= \partial_v F(u,v), &&\qquad \text{in $[0,\infty)$}\\
\bar u(0)&= \bar v(0)=0.
   \end{aligned}
\right.
\end{equation}
A posteriori, it then follows that $c>0$. Moreover, by Theorem~\ref{sec:limit-system-1}, it follows that $(\bar u,\bar v)$ is not stable in $\R$, so there exists $\varphi= (\varphi_1,\varphi_2) \in C^1_c((0,\infty),\R^2)$ such that
\begin{equation}
  \label{eq:not-stable-limit-1-caseii}
 \int_{\R} \Bigl(|\varphi'(t)|^2
-\bigl \langle D^2 F(\bar u,\bar v) \varphi,\varphi \bigr \rangle(t)
\Bigr)\,dt <0.
\end{equation}
Since
$$
\lim_{k \to \infty} (\bar u_k(t), \bar v_k(t))= (\bar u(t), \bar v(t))\qquad \text{and} \qquad
\lim_{k \to \infty} \sigma_k t  = 0
$$
uniformly in $t$ on the support of $\varphi$, we also have that
\begin{align*}
&\sigma_k^2 U_k(\sigma_k t)\\
&=  e^{- \sigma_k N t} D^2 F(\bar u_k,\bar v_k)(t) - \sigma_k^2 \beta_k^{2} \left (\begin{array}{cc}
W_1(\sigma_k t) & 0 \\
0 & W_2(\sigma_k t)
\end{array}
 \right)\\
&\longrightarrow\; D^2 F(\bar u,\bar v)(t) \qquad \text{uniformly in $t$ on the support of $\varphi$.}
\end{align*}
It then follows from (\ref{eq:not-stable-limit-1-caseii}) that
\begin{align*}
q_k:= \int_{0}^{\infty}&
e^{-\gamma_k \sigma_k t} |\varphi'(t)|^2\\
+& \sigma_k^2 \Bigl(\lambda  \beta_k^2 e^{-\gamma_k  \sigma_k t}|\varphi(t)|^2 -  \bigl \langle U_k(\sigma_k t) \varphi(t),\varphi(t)\bigr \rangle
\Bigr)\,dt <0
\end{align*}
for $k$ sufficiently large. Writing
$$
\varphi(t)= \psi(\sigma_k t) \qquad \text{with}\quad \psi= (\psi_1,\psi_2) \in C^1_c((0,\infty),\R^2),
$$
we then see, by a change of variable,
$$
q_k = \sigma_k \!\int_{0}^\infty\!\! \bigl(
e^{-\gamma_k \tau} |\psi'(\tau)|^2 + \lambda \beta_k^2 e^{-\gamma_k \tau}|\psi(\tau)|^2 -
\langle \psi(\tau),
U_k(\tau) \psi(\tau) \rangle
\bigr)d\tau = \sigma_k Q_k(\psi),
$$
contradicting the assumption (\ref{eq:stability-k}).\\
Since in both cases we reached a contradiction, we conclude that
$$
\text{$(u_k,v_k)$ remains bounded on compact subsets of $[0,\infty)$ as $k
\to \infty$.}
$$
Hence, by elementary ODE regularity estimates, we may pass to a subsequence such that
$$
u_k \to u_{0},  \quad v_k \to v_{0} \qquad \text{in
$C^1_{loc}([0,\infty)$ as ${\beta_k} \to 0$,}
$$
where $(u_0, v_0)$ is a nonnegative solution of the limit system
\begin{equation}
  \label{eq:henon-limit-system}
\left \{
 \begin{aligned}
- u_0''  &= e^{-Nt}\partial_{u}F(u_0,v_0)&&\qquad \text{in $(0,\infty)$,}\\
- v_0'' &= e^{-Nt}\partial_{v}F(u_0,v_0) &&\qquad \text{in $(0,\infty)$,}\\
u_0(0)&=v_0(0)=0.
    \end{aligned}
 \right.
\end{equation}
Moreover, passing to the limit in (\ref{eq:stability-k}), we find that
\begin{equation}
\int_{0}^\infty \Bigl(|\varphi'|^2  - e^{-Nt} \bigl \langle D^2 F(u_0,v_0) \varphi, \varphi \bigr \rangle \Bigr)\,dt \ge 0
 \label{stability-limit}
\end{equation}
for all $\varphi \in C_c^1((0,\infty),\R^2)$. Moreover, by Lemma~\ref{pohozaev-transformed} and Fatou's Lemma we have that
\begin{align}
&\int_0^\infty \bigl([u_0']^2 + [v_0']^2\bigr)\,dt = \liminf_{k \to \infty} \int_0^\infty (e^{-\gamma_k t} u_k)'^2 + (e^{-\gamma_k t} v_k)'^2 \,dt \nonumber\\
&\le \frac{2}{N} \lim_{k \to \infty} ([u_k'(0)]^2 + [v_k'(0)]^2) = \frac{2}{N} \bigl([u_0'(0)]^2 + [v_0'(0)]^2\bigr)< \infty.\label{limit-finite-energy}
\end{align}
Applying Lemma~\ref{approx-lemma}, we thus find a sequence $\varphi_n=(\varphi_n^1,\varphi_n^2)  \in C_c^1((0,\infty),\R^2)$ such that
$$
\lim_{n \to \infty} \int_0^\infty [(u_0-\varphi_n^1)'\,]^2\,dt = \lim_{n \to \infty} \int_0^\infty [(v_0-\varphi_n^2)\,]'^2\,dt= 0
$$
and
$$
\lim_{n \to \infty}e^{-\delta t}(u(t)-\varphi_n^1(t))=\lim_{n \to \infty}e^{-\delta t}(v(t)-\varphi_n^2(t))= 0
$$
uniformly in $t \ge 0$ for every $\delta>0$. Evaluating (\ref{stability-limit}) for $\varphi_n$ and passing to the limit, we thus see that
\begin{equation}
  \label{stability-limit-1}
\int_{0}^\infty \Bigl([u_0']^2 +[u_0']^2  - e^{-Nt} \Bigl \langle D^2 F(u_0,v_0) {u_0 \choose v_0} , {u_0 \choose v_0} \Bigr \rangle \Bigr)\,dt \ge 0.
\end{equation}
On the other hand, by Lemma~\ref{approx-lemma} we also have
$$
u_0^2(t)+v_0^2(t) \le C t \qquad \text{with}\; C:=\int_0^\infty \bigl([u_0']^2 + [v_0']^2\bigr)\,dt.
$$
Moreover, it follows from \eqref{limit-finite-energy} that there exist $t_n \ge 0$, $n \in \N$ with $t_n \to \infty$ and
$$
\sqrt{t_n}\bigl(|u_0'(t_n)|+|v_0'(t_n)|\bigr) \to 0 \qquad \text{as $n \to \infty$.}
$$
Consequently,
\begin{align}
&\int_0^\infty \bigl([u_0']^2 + [v_0']^2\bigr)\,dt =\lim_{n \to \infty}
\int_0^{t_n} \bigl([u_0']^2 + [v_0']^2\bigr)\,dt \nonumber\\
& =\lim_{n \to \infty}\Bigl(\frac{u_0(t_n)u_0'(t_n)}{2}-\int_0^{t_n} \bigl(u_0'' u_0 + v_0'' v_0 \bigr)\,dt\Bigr) \nonumber\\
& =\lim_{n \to \infty} \int_0^{t_n}e^{-Nt} \bigl(\partial_u F(u_0,v_0)u_0  + \partial_v F(u_0,v_0)v_0 \bigr)\,dt \nonumber \\
&=\int_0^{\infty}e^{-Nt} \bigl(\partial_u F(u_0,v_0)u_0  + \partial_v F(u_0,v_0)v_0 \bigr)\,dt. \nonumber
\end{align}
Combining this with (\ref{stability-limit-1}) and (\ref{F-consequence-3}), we deduce that
\begin{align*}
0 &\le \int_{0}^\infty e^{-Nt} \Bigl(\partial_u F(u_0,v_0)u_0  + \partial_v F(u_0,v_0)v_0-  \Bigl \langle D^2 F(u_0,v_0) {u_0 \choose v_0} , {u_0 \choose v_0} \Bigr \rangle \Bigr)\,dt\\
&\le - p(p-2)c_F \int_0^\infty e^{-Nt}\bigl(|u_0|^p+|v_0|^p\bigr)\,dt,
\end{align*}
which implies that $u_0 \equiv v_0 \equiv 0$. Consequently,
$$
\lim_{k \to \infty} \bigl([u_k'(0)]^2 + [v_k'(0)]^2\bigr) = \bigl([u_0'(0)]^2 + [v_0'(0)]^2\bigr) = 0
$$
yielding $u_k \equiv v_k \equiv 0$ for $k$ sufficiently large by Lemma~\ref{pohozaev-transformed} and the assumption $\lim \limits_{k \to \infty}\gamma_k \to 0$. The proof is finished.
\end{proof}

\section{Appendix}
\label{sec:appendix}

\subsection{Part A: A note on a linear ODE system on the half line}

The following Lemma is not surprising, as it relates variational instability of linear ODE system to the existence of negative eigenvalues and corresponding eigenfunctions.   
Since the proof is rather technical and not straightforward, we include it for the convenience of the reader. 

\begin{lemma}
\label{sec:morse-index-system-help-lemma}
Let $\delta> \gamma>0$, $\lambda \ge 0$, and let $U:[0,\infty) \to \R^{2 \times 2}$ be a bounded continuous function taking symmetric real $2 \times 2$-matrices as values. Suppose that there exists a function $\varphi \in C^1_c((0,\infty),\R^2)$ such that
$$
Q(\varphi):=\int_{0}^\infty \Bigl(e^{-\gamma t}|\varphi'(t)|^2 + \lambda e^{-\gamma t}|\phi(t)|^2  - e^{-\delta t} \bigl \langle U(t)\varphi(t),\varphi(t) \bigr \rangle\Bigr)\,dt <0,
$$
Then there exists $\mu<0$ and a function $h \in C^2([0,\infty),\R^2)$ such that
\begin{equation}
  \label{eq:radial-variable-linear-problem-1-help-lemma}
- \partial_t \bigl(e^{-\gamma t}\partial_t h\bigr) + \lambda e^{-\gamma t}h - e^{-\delta t} U(t) h
= \mu e^{-\delta t} h \quad \text{in
$[0,\infty)$}
\end{equation}
and such that
$$
h(0)=0, \qquad h(\infty)=\lim \limits_{t \to \infty}h(t)\; \text{exists.}
$$
\end{lemma}

\begin{proof}
In the following, $C>0$ always denotes a positive constant depending only on $\delta,\gamma$ and $U$. We introduce the weighted Sobolev space $H_*$ given as the completion of $C^1_c((0,\infty),\R^2)$ with respect to the norm $\|\cdot\|_*$ defined by
$$
\|h\|_*^2 = \int_{0}^\infty e^{-\gamma t}|h'(t)|^2\, dt
$$
Then $H_*$ is a Hilbert space. Moreover, for $h \in C^1_c((0,\infty),\R^2)$ we have, integrating by parts,
$$
\int_{0}^\infty e^{-\gamma t}|h(t)|^2\,dt = \frac{2}{\gamma}
\int_{0}^\infty e^{-\gamma t} \langle h(t), h'(t) \rangle \,dt
\le \frac{2}{\gamma}\Bigl(\int_{0}^\infty e^{-\gamma t}|h(t)|^2\,dt\Bigr)^{\frac{1}{2}}\|h\|_*
$$
and thus
$$
\|h\|_{L^2_\gamma}:= \Bigl(\int_{0}^\infty e^{-\gamma t}|h(t)|^2\,dt\Bigr)^{\frac{1}{2}} \le \frac{2}{\gamma}\|h\|_*.
$$
This estimate extends to functions on $H_*$ and shows that the quadratic form $Q$ is well defined on $H_*$. Moreover, for $h \in C^1_c((0,\infty),\R^2)$ we have the pointwise estimate
\begin{align*}
|h(t)|^2= 2 \int_0^t \langle h(s), h'(s) \rangle \,ds &\le 2 e^{\gamma t}
\int_0^t e^{-\gamma s} \langle h(s),  h'(s) \rangle \,ds\\
&\le 2 e^{\gamma t} \|h\|_{L^2_\gamma} \|h\|_* \le \frac{4}{\gamma} e^{\gamma t} \|h\|_*^2 \quad \text{for $t \ge 0$.}
\end{align*}
This pointwise estimate also extends to functions in $H_*$, and it easily shows that every $h \in H_*$ is continuous on $[0,\infty)$ with $h(0)=0$
and
\begin{equation}
  \label{eq:pointwise-estimate-h}
|h(t)| \le \frac{2}{\sqrt{\gamma}} \|h\|_*   e^{\frac{\gamma}{2}t}  \qquad \text{for $t \ge 0$.}
\end{equation}
Now by assumption we have
$$
\mu:= \inf_{M}Q \;<\;0 \qquad \text{for}\quad  M:= \Bigl \{\varphi \in H_*\::\: \int_0^\infty e^{-\delta t}|\varphi(t)|^2\,dt=1 \Bigr \}.
$$
Let $(h_n)_n \subset M$ be a sequence with $Q(h_n) \to \mu$ as $n \to \infty$. Since
$$
\int_0^\infty e^{-\delta t} \langle U(t)h_n(t),h_n(t) \rangle dt \le \|U\|_\infty  \int_0^\infty e^{-\delta t}|h_n(t)|^2\,dt =\|U\|_\infty \qquad \text{for $n \in \N$,}
$$
it follows that $\mu> -\infty$ and that $h_n$ is bounded in $H_*$. Passing to a subsequence,
we then have
\begin{equation}
  \label{eq:L2-loc-est}
h_n \weak h \quad \text{in $H_*$}\qquad \text{and}\qquad h_n \to h \quad \text{in $L^2_{loc}([0,\infty))$.}
\end{equation}
As a consequence of (\ref{eq:pointwise-estimate-h}), we also find that
\begin{equation*}
\int_{t}^\infty \!\!e^{-\delta s}h_n^2(s) ds \le \frac{4}{\gamma} \|h_n \|_*^2 \!\int_{t}^\infty \!\!e^{(\gamma-\delta)s}ds = \frac{4 }{\gamma (\delta -\gamma)} \|h_n \|_*^2 e^{(\gamma-\delta)t} \quad \text{for $t \ge 0$,}
\end{equation*}
where the RHS tends to zero as $t \to \infty$ uniformly in $n \in \N$. Combining this with (\ref{eq:L2-loc-est}), we conclude that
$$
\int_{0}^\infty e^{-\delta t}|h|^2 dt= \lim_{n \to \infty}\int_{0}^\infty e^{-\delta t}|h_n|^2 dt=1,
$$
which implies that $h \in M$. By the weak lower semicontinuity of $Q$, it then follows that $h$ is a minimizer of $Q$ on $M$. By a standard argument in the calculus of variations, this implies that $h$ is a classical solution of~(\ref{eq:radial-variable-linear-problem-1-help-lemma}) in the open interval $(0,\infty)$. Since, as remarked above, we also have $h \in
C([0,\infty),\R^2)$ and $h(0)=0$, we may use ~(\ref{eq:radial-variable-linear-problem-1-help-lemma}) to see that $h \in C^2([0,\infty),\R^2)$.
It remains to show that
\begin{equation}
  \label{eq:claim-limit-exist}
\text{the limit $\;h(\infty)= \lim_{t \to \infty} h(t)\;$ exists.}
\end{equation}
For this we need to distinguish two cases.\\
{\underline{Case 1:}} $\lambda>0$.\\
In this case we rewrite (\ref{eq:radial-variable-linear-problem-1-help-lemma}) as
\begin{equation}
  \label{eq:radial-variable-linear-problem-1-help-lemma-rewrite}
   h'' = \gamma h' +\lambda h - e^{(\gamma-\delta)t}V(t)h \qquad \text{with}\quad V(t)= \mu \, 1_{\R^{2\times 2}} + U(t).
\end{equation}
We then consider
$$
v: [0,\infty) \to \R, \qquad v(t)=|h(t)|^2= h_1^2(t)+ h_2^2(t),
$$
and we compute that
\begin{align*}
v''= 2 (|h'|^2 + \langle h'', h \rangle ) &= 2 \Bigl(|h'|^2 + \gamma \langle  h', h \rangle + \lambda |h|^2 + \langle e^{(\gamma-\delta)t}V(t) h, h\rangle\Bigr)\\
&= 2|h'|^2 + \gamma v' + 2 \Bigl(\lambda v - \langle e^{(\gamma-\delta)t}V(t) h, h\rangle\Bigr)\\
&\ge 2|h'|^2 + \gamma v' + 2 \Bigl(\lambda - e^{(\gamma-\delta)t} \|V\|_\infty \Bigr)v.
\end{align*}
Since $\gamma< \delta$, we may fix $t_0> 0$ such that $e^{(\gamma-\delta)t} \|V\|_\infty < \frac{\lambda}{2}$ for $t \ge t_0$, which yields that
$$
v'' -\gamma v' \ge 2|h'|^2 + \lambda v \ge 0 \qquad \text{on $[t_0,\infty)$.}
$$
With $\gamma_0:= \min \{2,\lambda\}>0$, we also have that
$$
2|h'|^2 + \lambda v \ge \gamma_0 (|h'|^2 + |h|^2) \ge \gamma_0 \langle h', h \rangle + \frac{\gamma_0}{2}(|h'|^2 + |h|^2)
\ge \gamma_0 v' + \frac{\gamma_0}{2}(|h'|^2 + |h|^2),
$$
and thus
\begin{equation}
  \label{eq:appendix-diff-ineq}
v'' - (\gamma+ \gamma_0) v' \ge \frac{\gamma_0}{2}(|h'|^2 + |h|^2) \ge 0 \qquad \text{on $[t_0,\infty)$.}
\end{equation}
Consequently, the function $v'- (\gamma+ \gamma_0)  v$ is increasing on $[t_0,\infty)$, and thus
$$
c:= \lim \limits_{t \to \infty}\bigl[v'(t)-(\gamma+ \gamma_0) v(t)\bigr] \; \in (-\infty,\infty]\qquad \text{exists.}
$$
We suppose by contradiction that $c \in (0,\infty]$. Then there exists $t_1 \ge t_0$ and $\rho>0$ such that
$$
v' - (\gamma+ \gamma_0) v \ge  \rho \qquad \text{on $[t_1,\infty)$}
$$
and thus
$$
\liminf_{t \to \infty} v(t)e^{-(\gamma+\gamma_0)t}>0,
$$
whereas (\ref{eq:pointwise-estimate-h})
  implies that
\begin{equation}
  \label{eq:pointwise-estimate-h-0-1}
v(t) \le \frac{4}{\gamma} \|h\|_*^2 \,  e^{\gamma t}  \qquad \text{for $t \ge 0$.}
\end{equation}
This is a contradiction, and thus $c \in (-\infty,0]$. Along a sequence $t_n \to \infty$ we then have, by (\ref{eq:appendix-diff-ineq}),
$$
0 = \lim_{n \to \infty}\Bigl(v''(t_n) -(\gamma+ \gamma_1) v'(t_n)\Bigr) \ge \frac{\gamma_0}{2} \limsup_{n \to \infty}\bigl( |h'(t_n)|^2+|h(t_n)|^2\bigr),
$$
which implies that $h(t_n) \to 0$ and $h'(t_n) \to 0$. Consequently, $v(t_n) \to 0$ and $v'(t_n)= 2 \langle h'(t_n), h(t_n) \rangle \to 0$ and thus
$$
c= \lim_{n \to \infty} \Bigl(v'(t_n) -(\gamma+\gamma_1) v(t_n)\Bigr) = 0
$$
Next we put $\tilde v(t):= v(t)e^{-(\gamma+\gamma_1)t}$, so that $\lim \limits_{t \to \infty} \tilde v(t)= 0$ by (\ref{eq:pointwise-estimate-h-0-1}).
As $c=0$, we find that
$$
\tilde v'(t) = \bigl(v'(t)-(\gamma+\gamma_1)v(t)\bigr)e^{-(\gamma+\gamma_1)t}= o(e^{-(\gamma+\gamma_1)t})\qquad \text{as $t \to \infty$}
$$
and hence
$$
|\tilde v(t)|= \Bigl| \int_t^\infty \tilde v'(s)\,ds \Bigr| = o(e^{-(\gamma+\gamma_1)t}) \qquad \text{as $t \to \infty$}
$$
which implies that
$$
v(t) = e^{(\gamma+\gamma_1)t}\tilde v(t) \to 0 \qquad \text{as $t \to \infty$.}
$$
Hence we conclude that $h(t) \to 0$ as $t \to \infty$, so (\ref{eq:claim-limit-exist}) holds.\\
{\underline{ Case 2:}} $\lambda=0$.\\
In this case we have, by (\ref{eq:radial-variable-linear-problem-1-help-lemma})
\begin{equation}
\label{equation-estimate-appendix}
|\partial_t \bigl(e^{-\gamma t} h'(t))| \le C
e^{- \delta t} |h(t)| \qquad \text{for $t \ge 0$}
\end{equation}
and thus
\begin{equation}
\label{equation-estimate-appendix-1}
|\partial_t \bigl(e^{-\gamma t} h'(t))| \le C e^{(\frac{\gamma}{2} - \delta) t} \qquad \text{as $t \to \infty$.}
\end{equation}
by (\ref{eq:pointwise-estimate-h}). Here and in the following, the letter $C$ denotes different positive constants. Since $\delta>\gamma > \frac{\gamma}{2}$, we infer that the limit $\lim \limits_{t \to \infty} e^{-\gamma t} h'(t)$ exists, and this limit must be zero since $\|h\|_*<\infty$.  We may then integrate
(\ref{equation-estimate-appendix-1}) to see that
$$
|e^{-\gamma t} h'(t)| \le C e^{(\frac{\gamma}{2} -\delta)t}
$$
and thus $|h'(t)| \le C e^{(\frac{3}{2}\gamma-\delta)t}$ for $t \ge 0$.
If $\frac{3}{2}\gamma-\delta \ge 0$, integration then shows that
$$
|h(t)|= \Bigl| \int_0^t h'(s)\,ds\Bigr| \le C e^{(\frac{3}{2}\gamma-\delta)t}\qquad \text{for $t \ge 0$.}
$$
We may then combine this estimate with (\ref{equation-estimate-appendix}) and integrate to find that
$$
|e^{-\gamma t} h'(t)|\le C  e^{(\frac{3}{2}\gamma-2\delta)t}\qquad \text{for $t \ge 0$,}
$$
i.e.,
$$
|h'(t)| \le C e^{(\frac{5}{2}\gamma-2\delta)t}\qquad \text{for $t \ge 0$.}
$$
We may iterate this process $n$ times, where $n \in \N$ is chosen such that
$\frac{\gamma}{2} + (n-1)(\gamma-\delta) \ge 0$ and
$\frac{\gamma}{2} + n(\gamma-\delta) < 0$. Consequently, we obtain that
$$
|h'(t)| \le C e^{(\frac{\gamma}{2} + n(\gamma-\delta))t} \quad \text{for $t \ge 0$.}
$$
Hence $h'(t)$ decays exponentially as $t \to \infty$, and from this we deduce (\ref{eq:claim-limit-exist}).\\
The proof is finished.
\end{proof}

\subsection{Part B: A doubling Lemma}
\label{sec:appendix-b:-doubling}

In the following, we note a simplified variant of a result of Polacik, Quittner and Souplet \cite{PQS}. We include the elementary proof for the convenience of the reader.

\begin{lemma}
\label{doubling-simplified-further} Let $(X,d)$ be a complete metric
space and $M: X \to (0,\infty)$ be a function which is bounded on compact subsets of
$X$. Then for any $s_* \in X$ there exists $t_* \in \overline{B_2(s_*)} \subset X$
such that
\begin{equation}
  \label{eq:doubling-property}
 M(t_*) \ge M(s_*) \qquad \text{and}\qquad M(t) \le 2 M(t_*)\quad \text{for all $t  \in B_{\frac{M(s_*)}{M(t_*)}}(t_*)$}.
\end{equation}
\end{lemma}

\begin{proof}
We follow the proof of Polacik, Quittner and Souplet \cite{PQS}.  Assuming by
contradiction that the lemma is false, we can successively construct a
sequence $(t_n)_n \subset X$ such that $t_0=s_*$
and
\begin{equation}
  \label{eq:doubling-property-proof}
M(t_{n+1}) > 2 M(t_{n}) \quad \text{and}\quad \dist(t_{n+1},t_{n})
\le M(s_*)/M(t_{n})
\end{equation}
for $n \in \N \cup \{0\}$. Indeed, suppose that $t_0,\dots,t_{n}$ are already constructed with this property. Then we have that
\begin{equation}
  \label{eq:doubling-property-proof-0}
M(t_k) \ge 2^{k} M(s_*) \qquad \text{for $k=0,\dots,n$}
\end{equation}
and thus
\begin{equation}
  \label{eq:doubling-property-proof-1}
\dist(t_k,t_{k-1}) \le \frac{M(s_*)}{M(t_{k-1})} \le 2^{1-k} \qquad \text{for $k=1,\dots,n$}.
\end{equation}
This shows that
$$
\dist(t_n,s_*) = \dist(t_n,t_0) \le \sum_{k=1}^n \dist(t_k,t_{k-1}) \le 2,
$$
So if there is no $t_{n+1}$ satisfying (\ref{eq:doubling-property-proof}),
then (\ref{eq:doubling-property}) is true with $t_*=t_{n} \in \overline{B_2(s_*)}$, contrary to what we assume. By induction, we thus see that the sequence exists as claimed.\\
On the other hand, this sequence is a Cauchy sequence by (\ref{eq:doubling-property-proof-1}), and $M(y_n) \to \infty$ as $n \to \infty$ by (\ref{eq:doubling-property-proof-0}). This contradicts the assumption that
$X$ is complete and $M$ is bounded on compact subsets.  Hence the lemma is proved.
\end{proof}

\end{document}